\documentclass{article}
\usepackage{fullpage}

\usepackage{amsmath, amssymb,bm}
\usepackage{hyperref}
\usepackage{enumerate,mathtools}
\usepackage{amsthm}
\usepackage{cite}
\usepackage{microtype}
\usepackage{color}
\usepackage{subcaption}

\newtheorem{theorem}{Theorem}
\newtheorem{cor}[theorem]{Corollary}
\newtheorem{lemma}[theorem]{Lemma}

\newtheorem{prop}[theorem]{Proposition}

\theoremstyle{definition}

\newtheorem{remark}{Remark}
\newtheorem{assumption}{Assumption}

\usepackage{pgfplots}
\usetikzlibrary{positioning}
\pgfplotsset{compat=1.12}

\newcommand{\cF}{\mathcal{F}}

\newcommand{\dd}{\mathrm{d}}

\newcommand{\ex}[1]{\ensuremath{\mathbb{E}\left[ #1\right]}}

\DeclareMathOperator{\mmse}{\sf mmse}

\DeclareMathOperator{\snr}{\sf snr}

\newcommand{\MSE}{\mathsf{MSE}}
\newcommand{\AMP}{\mathsf{AMP}}
\DeclareMathOperator{\MMSE}{\mathsf{MMSE}}

\newcommand{\reals}{\mathbb{R}}

\newcommand{\eps}{\epsilon}
\newcommand{\normal}{\mathcal{N}}

\usepackage{xspace}
\newcommand{\iid}{i.i.d.\xspace}
\newcommand{\iiddistr}{{\stackrel{\text{\iid}}{\sim}}}
\renewcommand{\hat}{\widehat}
\renewcommand{\tilde}{\widetilde}
\newcommand{\Bern}{{\rm Bern}}

\let\originalleft\left
\let\originalright\right
\renewcommand{\left}{\mathopen{}\mathclose\bgroup\originalleft}
\renewcommand{\right}{\aftergroup\egroup\originalright}

\allowdisplaybreaks

\newif\iflongpaper

\longpapertrue

\title{All-or-Nothing Phenomena:\\ From Single-Letter to High Dimensions}

 \author{
 { Galen Reeves} \thanks{G.~Reeves is with the Department of Electrical and Computer Engineering and the Department of Statistical Science, Duke University, Durham, NC 27708 USA; e-mail: {\tt galen.reeves@duke.edu}.  G. Reeves is  is supported by NSF Grants CCF-1718494 and  CCF-1750362.}
 \and
 { Jiaming Xu}\thanks{J. Xu is with the Fuqua School of Business, Duke University, Durham, NC 27708 USA; e-mail: {\tt jiamingxu.868@duke.edu}.  J. Xu is  is supported by NSF Grants IIS-1932630, CCF-1850743, and CCF-1856424.} 
 \and
 { Ilias Zadik}
 \thanks{I. Zadik is with the Center for Data Science, New York University, New York, NY 10011 USA; e-mail: {\tt zadik@nyu.edu}. I. Zadik is supported by a CDS-Moore-Sloan postdoctoral fellowship. 
 \newline
 \indent This paper was presented at the 2019 IEEE International Workshop on Computational Advances in Multi-Sensor Adaptive Processing (CAMSAP), Guadeloupe, French West Indies. 
 }
}

\begin{document}


\maketitle

\maketitle

\begin{abstract} 
We consider the linear regression problem of estimating a $p$-dimensional vector $\beta$ from $n$ observations $Y = X \beta + W$, where $\beta_j \iiddistr \pi$ for a real-valued distribution $\pi$ with zero mean and unit variance,  $X_{ij} \iiddistr \normal(0,1)$, and $W_i\iiddistr \normal(0, \sigma^2)$.  In the asymptotic regime where  $n/p \to \delta$  and $ p/ \sigma^2 \to \snr$ for two fixed constants $\delta, \snr\in (0, \infty)$
as $p \to \infty$,  the limiting (normalized) minimum mean-squared error (MMSE) has been characterized by the MMSE of an associated single-letter (additive Gaussian scalar) channel.

In this paper, we show that if the MMSE function of the single-letter channel converges to a step function, then the limiting MMSE of estimating $\beta$ in the linear regression problem converges to a step function which jumps from $1$ to $0$ at a critical threshold.  Moreover, we establish that the limiting mean-squared error of the (MSE-optimal) approximate message passing algorithm also converges to a step function with a larger threshold, providing evidence for the presence of a computational-statistical gap between the two thresholds.
\end{abstract}

\section{Introduction} 
Consider the classical linear regression  model
\begin{align}\label{reg}
Y = X \beta + W
\end{align}
where  $X \in \reals^{n\times p}$ with $X_{ij}\iiddistr \normal(0,1)$, $\beta \in \reals^p$ with $\beta_j \iiddistr \pi$ for a distribution $\pi$ with zero mean and unit variance, and $W \in \reals^n$ with $W_{i}\iiddistr \normal(0,\sigma^2)$. 
We are interested in estimating $\beta$ from
observation of $(X,Y).$ For a given estimator 
$\hat \beta (X,Y)$, the normalized mean squared-error of estimating $\beta$ is given by
$$
\MSE\left(\hat\beta \right) :=
\frac{1}{p} \ex{\left\| \beta - \hat{\beta}\right \|^2}.
$$
Let $\MMSE$ denote the minimum of $\MSE\left(\hat\beta \right)$ among all possible estimators $\hat\beta$, or equivalently, 
\begin{align}
\MMSE := \frac{  1 }{  p} \ex{\left\| \beta - \ex{\beta \mid X,Y } \right \|^2}.
\end{align}
In this paper, we focus on the asymptotic regime: 
\begin{align}
\frac{n}{p} \to \delta \quad \text{ and } \quad 
\frac{p}{\sigma^2} \to \snr, \quad \text{ as } p \to \infty,\label{eq:limit_p}
\end{align}
for two fixed constants $\delta, \snr \in (0,\infty)$. 
Note that $\delta$ is the under-sampling ratio and $\snr$ is the signal-to-noise ratio in view of $\ex{ \| X  \beta\|^2}/ \ex{ \|W\|^2} = p / \sigma^2 $.

Recent work~\cite{GuoVerdu,reeves:2019c,barbier:2016} proves that under certain structural assumptions in terms of $(\pi,\delta,\snr)$, the limiting $\MMSE$ in
the asymptotic regime \eqref{eq:limit_p} is characterized by the \emph{replica-symmetric}  (RS) formula through a single letter channel
\begin{align}
y=\sqrt{s}\beta_0+N, \label{eq:single_letter}
\end{align}
where $s>0$, $\beta_0 \sim \pi$ and $N\sim \normal(0,1)$ are independent.  However, often the RS formula is  too complicated to extract structural behavior of the limiting MMSE. 

In this work, we propose generic conditions under which the limiting $\MMSE$ exhibits an all-or-nothing phenomena. More precisely, consider a family $(\pi_\eps, \delta_\eps,\snr_\eps)$ indexed by a positive parameter $\eps$ where $\pi_\eps$ has finite entropy $H_\eps:=H(\pi_\eps)$. We show that if the MMSE of the  single letter channel \eqref{eq:single_letter} as a function of $s$ converges to a step function as $\eps \to 0$, then the limiting $\MMSE$ of the linear regression model \eqref{reg} also converges to a step function, which jumps from $1$ to $0$
at a critical threshold $\delta_\eps=\delta_{\eps,\MMSE}$, where
\begin{align}
\delta_{\eps,\MMSE}:= \frac{ 2H_\eps}{  \log(1  + \snr_\eps)}. \label{MMSEthr0}
\end{align}
 In other words, an all-or-nothing phenomena in the single letter channel implies an all-or-nothing phenomena in the high-dimensional linear regression model. Moreover, we establish that the limiting $\MSE$ of the (MSE-optimal) approximate message passing (AMP) algorithm  also converges to a step function,
which jumps from $1$ to $0$ at a larger threshold $\delta_\eps=\delta_{\eps,\AMP}$, where
\begin{align}
\delta_{\eps,\AMP }:= \frac{ 2H_\eps(1+\snr_\eps)}{ \snr_\eps}. \label{Algthr0}
\end{align}

An important application of our general result is the binary linear regression model where $\beta_j \iiddistr \Bern(\eps)$. In this case, we show that the MMSE function of the single letter channel converges to a step function as  the sparsity $\eps \to 0$. Then we obtain from our general result that the limiting $\MMSE$ of the binary linear regression model
converges to a step function which jumps from $1$ to $0$ at the critical threshold $\delta_\eps= \frac{2\eps \log (1/\eps)}{\log (1+ \snr_\eps) } $. This coincides with the all-or-nothing phenomenon established in
\cite{reeves:2019b} for
the binary linear regression model where 
$\beta$ is chosen uniformly at random from the set of binary $k$-sparse vectors,
in the highly sparse and high signal-to-noise ratio regime  where $k/\sqrt{p}\to 0$ and  $k/\sigma^2$ is above a sufficiently large constant. Furthermore, we deduce from our general result that the limiting $\MSE$ of the (MSE-optimal) AMP converges to a step function which jumps from $1$ to $0$ at the critical threshold $\delta_\eps= \frac{2\eps \log (1/\eps) (1+\snr_\eps)}{\snr_\eps } $. 
This coincides with the computational threshold for a number of computationally efficient methods in the literature such as LASSO or Orthogonal Matching Pursuit. In particular, our result adds to the existing evidence for the presence of a computational-statistical gap between the two thresholds (see \cite{Zadik17a,Zadik17b} for an extended discussion and literature review on the presence of this computational-statistical gap).


\section{Preliminaries}

\subsection{The Replica Symmetric Formulas}

To describe the RS formulas, we first define the mutual information and MMSE functions for the single letter channel \eqref{eq:single_letter}:
\begin{align}
I(s) & := I(\beta_0; \sqrt{s} \beta_0  + N), \quad s > 0 \label{MIS} \\
M(s) & := \mmse(\beta_0 \mid \sqrt{s} \beta_0 + N),  \quad s > 0 \label{MSES}
\end{align} 
where $\beta_0 \sim \pi$ and $N \sim  \normal(0, 1)$ are independent. Both of these functions are non-negative and the unit variance assumption on $\pi$ means that for any $s > 0,$ (see \cite{GuoVerduShamai} for details)
\begin{align}
I(s)  & \le  \frac{1}{2} \log(1 + s) \le \frac{s}{2}, \\
M(s) & \le \frac{1}{ 1 + s}  \le 1.
\end{align}
Moreover, the I-MMSE relation for the single-letter channel~\cite{GuoVerduShamai} states the derivative of the mutual information is one half the MMSE, that is $I'(s) = \frac{1}{2} M(s)$.  

Next, we define the potential function $\cF : [0,\infty) \to [0,\infty)$ according to
\begin{align}
\cF(s) & := I (s )  + \frac{\delta}{2} \phi\left( \frac{s   }{\delta \snr }   \right), \label{pote}
\end{align}
where $ \phi(x)  = x - \log x - 1$, and $\delta, \snr$ are respectively the undersampling ratio and the signal-to-noise ratio of our original model. Note that $\phi(x)$ is convex and non-negative on $(0, \infty)$.

\begin{lemma}\label{lem:cFmin} 
All stationary points of  $\cF(s)$ lie on the open interval between $\delta \snr / (1 + \snr)$ and $\delta \snr$. 
\end{lemma}
\begin{proof} 
By differentiation with respect to $s$ and the I-MMSE relation for the single-letter channel~\cite{GuoVerduShamai}, 
 we have that for any $s>0$ $$\cF'(s)  \propto M(s)   +  1/ \snr-  \delta / s .$$  The fact that $M(s) > 0$ for all $s$  implies that $\cF'(s)$ is strictly positive  for all $s \ge \delta \snr$, and thus $\cF(s)$ is strictly increasing on $[\delta \snr, \infty)$. Alternatively, the fact that $M(s) < 1$ for all $s > 0$ implies that $\cF'(s)$ is strictly negative for all $s \le \delta \snr / (1 + \snr)$, and thus $\cF(s)$ is strictly decreasing on $(0, \delta \snr  / (1 + \snr)]$. 
\end{proof}

In view of Lemma~\ref{lem:cFmin}, the minimum of the potential function and the smallest and largest minimizers can be defined as follows: 
\begin{align}
\cF^* &:=  \min_{s} \cF(s), \\
\underline{s}^*& := \min\{ s \, : \, \cF(s) = \cF^*\},\\
\overline{s}^* &:= \max\{ s \, : \, \cF(s) = \cF^*\}.
\end{align}
Note that $\underline{s}^* = \overline{s}^*$ if and only if the minimum is attained at a unique point. 

\begin{prop}[RS MMSE\cite{reeves:2019c, barbier:2016,BarbierPNAS}]\label{RS_MMSE}
For any $(\delta,\snr,\pi)$ for which $(\snr,\pi)$ satisfies the single-crossing property \cite{reeves:2019c} and $\pi$ has finite fourth moment \footnote{A different set of assumptions on  $(\delta,\snr,\pi)$ for which the Proposition \ref{RS_MMSE} holds can be found in \cite{barbier:2016,BarbierPNAS}}, the mutual information and MMSE satisfy 
\begin{align}
\lim_{p \to \infty} \frac{  1 }{  p} I(\beta; X, Y) & = \cF^*, \\
\limsup_{p \to \infty} \frac{  1 }{  p} \ex{\left\| \beta - \ex{\beta \mid y , X } \right \|^2}&\le M(\overline{s}^* ),  \\
\liminf_{p \to \infty} \frac{  1 }{  p} \ex{\left\| \beta - \ex{\beta \mid y , X } \right \|^2}&\ge M(\underline{s}^* )  ,
\end{align}
where the limits are taken as $(n=n_p,p,\sigma^2=\sigma^2_p)$ scale to infinity with $p \to +\infty,$ $n/p \to \delta$ and $p/\sigma^2 \to \snr$. 
\end{prop}

Next, we turn to the family of \emph{approximate message passing} (AMP) \cite{donoho:2009a,bayati:2011} algorithms and specifically to the case of MMSE-AMP which is proven in to be optimal among AMP algorithms in minimizing the MSE of the recovery problem of interest \cite{Reeves:2012}. For simplicity from now on when we say AMP we refer to the MMSE-AMP algorithm. It turns out that a related formula to the one given in Proposition \ref{RS_MMSE} describes the asymptotic MSE associated with AMP.

The smallest stationary point is defined as
\begin{align}
s^\AMP  & := \inf\{ s  \, : \, \cF'(s) = 0 \}.
\end{align}
It is rather straightforward to check that $s^\AMP$ is attained by some positive value $s$ and therefore it is a stationary point of $\cF(s)$. In particular, by Lemma~\ref{lem:cFmin} we have $s^\AMP   \in  ( \delta \snr /(1 + \snr), \delta \snr)$. 

For the next result, for $T \in \mathbb{N}$ let $\hat{\beta}_{\AMP,T}(Y,X)$ be  the output of the AMP estimator \cite[Section II.C]{Reeves:2012} with input data $(Y,X)$ after $T$ iterations. 
\begin{prop}[AMP, \cite{bayati:2011, Reeves:2012}]
\label{propAMP}
For any $(\delta,\snr,\pi)$ where $\pi$ has a finite fourth moment,  AMP satisfies
\begin{align}
\lim_{T \rightarrow + \infty} \lim_{p \rightarrow + \infty} \frac{  1 }{  p} \ex{\left\| \beta - \hat{\beta}_{\AMP,T}(Y,X)  \right \|^2}& =  M(s^\AMP )  
\end{align}
where the inner limit is taken as $(n=n_p,p,\sigma^2=\sigma^2_p)$ scale to infinity with $p \to +\infty,$ $n/p \to \delta$ and $p/\sigma^2 \to \snr$. 
\end{prop}

\begin{remark}
The results stated above imply that AMP is optimal whenever $\underline{s}^* = \overline{s}^* = s^{\AMP} $. 
\end{remark}

\begin{remark}
For a proof of Proposition \ref{propAMP} we refer the reader to the statement and proof of \cite[Theorem 6]{Reeves:2012}.
\end{remark}
\section{Main Results}

Let us consider now a family of coefficient distributions $(\pi_\eps)_{\eps>0}$ indexed by a positive-valued parameter $\eps>0$. We assume throughout the section that for each $\epsilon>0$ the distributions $\pi_\eps$ has zero mean, unit variance  and finite entropy $H_\eps$. 
Our results are all based on the following assumption on the family $\pi_\eps.$
\begin{assumption}\label{dfn:stepass}
Let $(\pi_\eps)_{\eps>0}$ be a family of distributions with unit variance and finite entropy $H_\eps$. The MMSE function $M_\eps(s)$ of the single letter channel, as defined in (\ref{MSES}), for $\pi_\eps$ coefficient distribution is assumed to converge pointwise to a step function as $\eps \to 0$ in the following sense
\begin{align}
\lim_{\eps \to 0} M_\eps( 2 H_\eps\,  t) &= \begin{cases}
1 , & t \in [0,1) \\
0, & t \in (1, \infty).
\end{cases}\label{eq:Meps_cond}
\end{align}
\end{assumption} 

\begin{remark}
It can be straightforwardly checked using the I-MMSE relation for the single-letter channel \cite{GuoVerduShamai} that the rescaling in the argument of $M_\eps$ by twice the entropy term, i.e.\ by $2 H_\eps  $, is necessary for the convergence of $M_{\eps}$ to the step function. To see why, observe that $M_{\eps}(s)$, using the I-MMSE relation, satisfies the integral constraint $  \int_0^{ \infty} M_\eps ( 2 H_\eps  t ) \, d t  = 1$. Thus, convergence to a step function at a threshold other than $2 H_\eps$ would violate this constraint.
\end{remark}

\begin{remark}
As we establish later in the section, Assumption~\ref{dfn:stepass} is satisfied for the family of (normalized) Bernoulli distributions with probability $\epsilon.$ 
See Fig.~\ref{fig:single_letter_bernoullie} for a graphical illustration. We expect the assumption to hold under greater generality. 
\end{remark}
We now present our two main results assuming the family of distributions $( \pi_\eps)_{\eps>0}$ satisfies Assumption \ref{dfn:stepass}. 

\begin{figure}
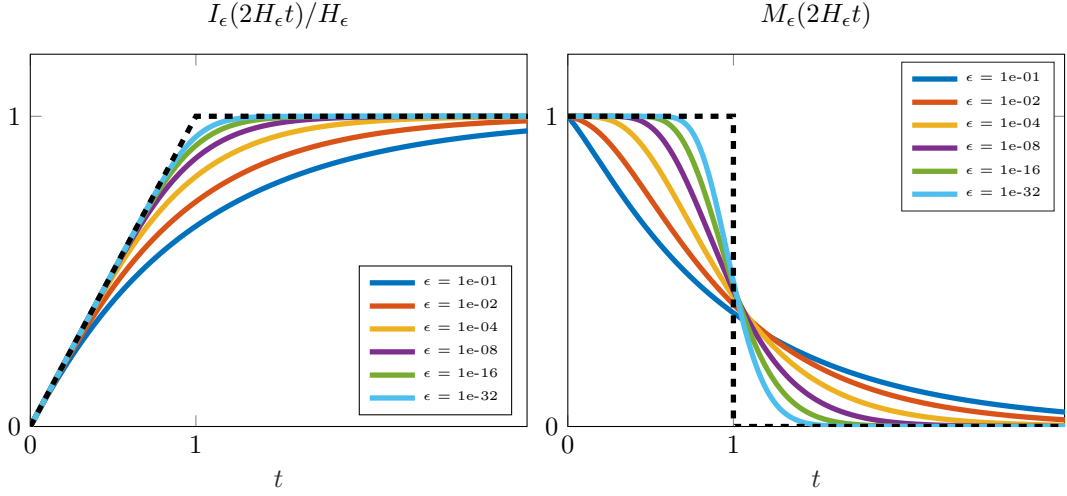

\centering
\input{Is.tex} \input{Ms.tex} 
\caption{Single letter mutual information and MMSE functions corresponding to the Bernoulli$(\eps)$ distribution normalized to unit variance.}
\label{fig:single_letter_bernoullie}
\end{figure}

\begin{theorem}\label{thm:MMSE_Thm}
Let $( \pi_\eps)_{\eps>0}$ satisfy Assumption \ref{dfn:stepass}. Given a number $r \in (0,1) \cup (1, \infty)$, let $ (\delta_\eps, \snr_\eps, \pi_\eps)_{\eps>0}$ be a family of triplets such that
\begin{align}
\lim_{\eps \to 0} \frac{\delta_\eps  }{ \delta_{\eps,\MMSE}  } &= r. \label{eq:scaling_r}
\end{align}
Then, the minimizers of the RS potential function exhibit the all-or-nothing behavior in the small-$\eps$ limit depending on whether $r$ is greater than or less than one: 
\begin{align}
r \in (0, 1) \quad \implies \quad  & \lim_{\eps \to 0} M_\eps \left( \overline{s}^*_\eps \right )  = 1\\
r \in (1, \infty) \quad \implies \quad  & \lim_{\eps \to 0}  M_\eps \left(  \underline{s}^*_\eps\right )  =0.
\end{align}
\end{theorem}

Combining Theorem \ref{thm:MMSE_Thm} with Proposition \ref{RS_MMSE} we obtain the following Corollary.
\begin{cor}[All-or-nothing MMSE behavior]\label{cor:MMSE}
Let $r \in (0,1) \cup (1, \infty)$. For any family of triplets $ (\delta_\eps, \snr_\eps, \pi_\eps)_{\eps>0}$, 
suppose that for any $\eps>0$, $(\snr_\eps, \pi_\eps)$ satisfies the single-crossing property, $\pi_\eps$ has finite fourth moment, $( \pi_\eps)_{\eps>0}$ satisfies Assumption \ref{dfn:stepass}, and \eqref{eq:scaling_r} holds.
Then it holds that 
\begin{align}
\lim_{\eps \to 0} \lim_{p \to \infty} \frac{  1 }{  p} \ex{\left\| \beta - \ex{\beta \mid X , Y } \right \|^2} = \begin{cases}
1 , & r \in [0,1) \\
0, & r \in (1, \infty).
\end{cases}
\end{align}where the inner limits are taken as $(n=n_p,p,\sigma^2=\sigma^2_p)$ scale to infinity with $p \to +\infty,$ $n/p \to \delta_\eps$ and $p/\sigma^2 \to \snr_\eps$.
\end{cor}
We next present our second main result.

\begin{theorem}\label{thm:AMP_Thm} 
Let $( \pi_\eps)_{\eps>0}$ satisfy Assumption \ref{dfn:stepass}. Given a number $r \in (0,1) \cup (1, \infty)$, let $ (\delta_\eps, \snr_\eps, \pi_\eps)$ be such that
\begin{align}
\lim_{\eps \to 0} \frac{\delta_\eps  }{ \delta_{\eps,\AMP}} &= r .  \label{eq:scaling_rho}
\end{align}
Then, the smallest stationary point $s^\AMP$ exhibits the all-or-nothing behavior in the small-$\eps$ limit depending on whether $r$ is greater than or less than one: 
\begin{align}
r \in (0, 1) \quad \implies \quad  & \lim_{\eps \to 0} M_\eps \left( s^\AMP_\eps  \right )  = 1\\
r \in (1, \infty) \quad \implies \quad  & \lim_{\eps \to 0} M_\eps \left( s^\AMP_\eps\right )  = 0.
\end{align}
\end{theorem}

For the result, for $T \in \mathbb{N}$ let $\hat{\beta}_{\AMP,T}(Y,X)$ the output of the AMP estimator \cite[Section II.C]{Reeves:2012} with input data $(X,Y)$ after $T$ iterations. Combining Theorem \ref{thm:AMP_Thm} with Proposition \ref{propAMP} we obtain the following Corollary on the performance of AMP.

\begin{cor}[All-or-nothing AMP behavior]\label{cor:AMP}
Let $r \in (0,1) \cup (1, \infty)$. For any family of triplets $(\delta_\eps,\snr_\eps,\pi_\eps)_{\eps>0}$, suppose that each $\pi_{\eps}$ has a finite fourth moment, the family $(\pi_\eps)_{\eps>0}$ satisfies Assumption \ref{dfn:stepass}, and 
\eqref{eq:scaling_rho} holds.
Then it holds that 
\begin{align}
\lim_{\eps \to 0} \lim_{T \rightarrow +\infty} \lim_{p \to \infty} \frac{  1 }{  p} \ex{\left\| \beta -  \hat{\beta}_{\AMP,T}(X,Y)  \right \|^2}  =\begin{cases}
1 , & r \in [0,1) \\
0, & r \in (1, \infty).
\end{cases}
\end{align}
where the inner limits are taken as $(n=n_p,p,\sigma^2=\sigma^2_p)$ scale to infinity with $p \to +\infty,$ $n/p \to \delta_\eps$ and $p/\sigma^2 \to \snr_\eps$.
\end{cor}

\subsection{Illustration of Theorems~\ref{thm:MMSE_Thm} and \ref{thm:AMP_Thm} via normalized potential function}

To provide insight into the behavior described by Theorems~\ref{thm:MMSE_Thm} and \ref{thm:AMP_Thm}, we consider the scaling $\delta_{\eps} = 2 r H_\eps  /  \log(1 + \snr)$ where the pair $(\snr,r)$ is considered fixed with respect to $\eps>0$. Note that in this scaling, $\delta_{\eps}=r \delta_{\eps,\mathrm{MMSE}}$. Define the normalized potential function
\begin{align}
\tilde{\cF}_{\eps,r}(t) &: = \frac{ \cF ( 2 H_\eps t)}{ H_\eps}   =  \frac{ I_\eps ( 2 H_\eps t)}{ H_\eps}  + \frac{r }{ \log(1 + \snr)}  \phi\left(  \frac{ t  \log(1 + \snr)}{ r \snr}  \right) .
\end{align}
Note that only the first term depends on $\eps$. Under Assumption~\ref{dfn:stepass}, using
the I-MMSE relation we have $I_\eps ( 2 H_\eps t) /  H_\eps  \to  1 \wedge t$ as $\eps \to 0$. Thus for all $t>0$, \begin{align}\label{limitF} \lim_{\eps \to 0} \tilde{\cF}_{\eps,r}(t) = \tilde{\cF}_{0,r}(t), \end{align} where 
\begin{align}
\tilde{\cF}_{0,r}(t) &:=  1 \wedge t + \frac{r }{ \log(1 + \snr)}  \phi\left(  \frac{ t  \log(1 + \snr)}{ r \snr}  \right). 
\end{align}
For $r \in (0,1) \cup (1, \infty)$,  the function $\tilde{\cF}_{0,r}(t)$ has a unique global minimizer given by
\begin{align}
t^* &=
\begin{dcases}
\frac{r \snr }{ (1 + \snr) \log(1 + \snr)}, & r \in (0,1) \\
 \frac{r\snr}{ \log(1 + \snr)}  & r \in (1, \infty).
\end{dcases}
\end{align}
Importantly, using the elementary inequality $x/\left(x+1\right) \leq \log \left(1+x\right) \leq x$ for all $x>-1$, we deduce
\begin{align*} 
r \in (0, 1) \quad \implies \quad  & t^*   \in (0,1) \\
r \in (1, \infty) \quad \implies \quad  & t^*  \in (1,\infty).
\end{align*}This dichotomy together with Assumption~\ref{dfn:stepass} and equality \eqref{limitF} suggest the following all-or-nothing behavior; as $\epsilon \to 0$, when $r<1$ the MMSE converges to one and when $r>1$ the MMSE  converges to zero. This is the same all-or-nothing behavior  described and rigorously established in Theorem~\ref{thm:MMSE_Thm}. Furthermore, the smallest stationary point $t^\AMP$ is given by 
\begin{align}
t^\AMP&=
\begin{dcases}
\frac{r \snr }{ (1 + \snr) \log(1 + \snr)}, & r \in (0,r^\AMP) \\
 \frac{r\snr}{ \log(1 + \snr)} , & r \in (r^\AMP, \infty),
\end{dcases}
\end{align}where $r^\AMP: = \delta_{\eps, \AMP} / \delta_{\eps, \MMSE} = (1 + 1/\snr) \log(1+\snr)$. Similar to above, we deduce
\begin{align*} 
r \in (0, r^\AMP) \quad \implies \quad  & t^\AMP  \in (0,1) \\
r \in (r^\AMP, \infty) \quad \implies \quad  & t^\AMP  \in (1,\infty).
\end{align*}This dichotomy together with Assumption~\ref{dfn:stepass} and equality \eqref{limitF}, suggest the following all-or-nothing behavior; as $\epsilon \to 0$, when $r<r^\AMP$ the MSE of the AMP converges to one and when $r>r^\AMP$ the MSE of the AMP converges to zero. This is the  same all-or-nothing behavior  described and rigorously established in Theorem~\ref{thm:AMP_Thm}.

This behavior of the normalized potential function is illustrated graphically in Figure~\ref{fig:pot_funct} where both $\tilde{\cF}_{\eps,r}(t)$ and $\tilde{\cF}_{0,r}(t)$ are plotted as a function of $t$ for various $r$.

\begin{figure}
\centering
\input{potential_small_eps.tex} \input{potential_zero_eps.tex} 
\caption{\label{fig:pot_funct}Illustration of the normalized potential functions, 
$\tilde{\cF}_{\eps,r}(t)$ for $\eps=10^{-16}$ (top plot) and $\tilde{\cF}_{0,r}(t)$(bottom plot), where $\snr=5$ and $r$ varies in $(0,+\infty)$. Black-colored curves correspond to $r \in (0,1)$, magenta-colored curves correspond to $r \in (1,r^\AMP)$ and cyan-colored curves correspond to $r \in (r^\AMP,\infty)$. Note that the global minimizer of $\tilde{\cF}_{0,r}(t)$ transitions from being less than one to bigger than one exactly at $r=1$, while the smallest stationary point transitions from being less than one to bigger than one exactly at $r=r^{\AMP}$. A similar approximate behavior takes place for $\tilde{\cF}_{\eps,r}(t)$ with $\eps=10^{-16}$. 
}
\end{figure}

\section{Application: Sparse binary regression}

We now present our main application of our two results to sparse binary regression, where $\beta_i \iiddistr \Bern(\eps)$. To this end, we first consider the case where
$\beta_i$ is i.i.d.\ drawn from the following 
two-point distribution:
\begin{align}
\pi_\eps = (1 - \eps)\,  \delta_{\mu_1}  + \eps \, \delta_{ \mu_2}, \label{eq:pi_two_point} 
\end{align}
where  $\delta_x$ denotes a Dirac distribution with mass at $x \in \reals$, and $\mu_ 1 = - \sqrt{ \eps /(1- \eps)}$ and $\mu_ 2 =  \sqrt{ (1- \eps)/\eps}$ are chosen such that $\pi_\eps $ has zero mean and unit variance. The following Lemma holds for the family of MMSE functions $(M_\eps(s))_{\eps >0}:$
 
\begin{lemma}\label{lem:two_point}
The distribution $\pi_\eps$ in \eqref{eq:pi_two_point} has entropy $H_\eps = - \eps \log \eps  - (1-\eps) \log(1-\eps)$ and MMSE function 
\begin{align}
M_\eps(s) 
& =    \ex{ \frac{1}{  1  - \eps +   \eps  \,   \exp\left(   \frac{ s}{2  \eps (1- \eps) } +    \sqrt{\frac{s}{ \eps (1-\eps)} }  N   \right)}},
\end{align}
where $N \sim \normal(0,1)$. Furthermore, the distribution $\pi_\eps$ satisfies the single-crossing condition \cite{reeves:2019c} for all $\snr >0$ and
\begin{align}
\lim_{\eps \to 0} H_\eps /(\eps \log 1/\eps) &= 1\\
\lim_{\eps \to 0} \sup_{s >  0 }  \left |  M_\eps ( s) - Q\left(  \frac{ s- 2 \eps \log(1/\eps)}{ 2 \sqrt{s \eps} } \right)  \right | &= 0, 
\end{align}
where $Q(z) = \int_z^\infty (2\pi)^{-1/2} \exp(- t^2/2) \, d t$. 
\end{lemma}

An immediate implication of the result is that the family of distributions $(\pi_\eps)_{\eps>0}$ satisfies Assumption~\ref{dfn:stepass} as well as the conditions of Corollaries \ref{cor:MMSE} and \ref{cor:AMP}. 
Hence, all-or-nothing phase transitions hold for the limiting 
$\MMSE$ around $\delta_{\eps, \MMSE}$  and for the MSE of the AMP around $\delta_{\eps, \AMP}.$ Using that $\lim_{\eps \to 0} H_\eps /(\eps \log 1/\eps) = 1$, we can further simplify the phase transition points given in \eqref{MMSEthr0}, \eqref{Algthr0} by observing \begin{align} \label{MMSEsparse1}\lim_{\eps \to 0} \delta_{\eps, \MMSE} / \left( \frac{ 2 \eps \log(1/\eps)}{ \log(1 + \snr_\eps)} \right)=1 \end{align}and
\begin{align} \label{AMPsparse1}\lim_{\eps \to 0} \delta_{\eps, \AMP} / \left( \frac{ 2(1+\snr_{\eps}) \eps \log(1/\eps)}{ \snr_\eps} \right)=1. \end{align}

Next, we extend the above results to the sparse binary regression problem of interest, where  $\beta_i \iiddistr \Bern(\epsilon)$.  We denote by $k=\eps p$ the (expected) number of non-zero coordinates of $\beta$. 
Define
$$
\tilde{\beta} = \frac{\beta-\ex{\beta}}{\sqrt{\eps(1-\eps)}}.
$$
Then $\tilde{\beta}_i \iiddistr \pi_\eps$ as given in \eqref{eq:pi_two_point}.
Moreover, define
$$
\tilde{Y}= \frac{Y-X\ex{\beta}}{
\sqrt{\eps(1-\eps)}
}, \quad
\tilde{W} = \frac{W}{ \sqrt{\eps(1-\eps)} }.
$$
Then it follows that 
$\tilde{Y}=X\tilde{\beta} + \tilde{W}$.
Since 
$\tilde{W}_i \iiddistr \normal(0,
\tilde{\sigma}^2)$ with $\tilde{\sigma}=\sigma/\sqrt{\eps(1-\eps)}$, it follows that 
\begin{align}
\snr_\eps = \frac{p}{\tilde{\sigma}^2}
=\frac{ p \eps (1-\eps) }{\sigma^2}.
\label{eq:snr_binary}
\end{align}

Hence, according to Corollary \ref{cor:MMSE}, (\ref{MMSEsparse1}),
\eqref{eq:snr_binary},
we obtain that the limiting $\MMSE$ exhibits an all-or-nothing behavior at $$\delta_{\eps, \MMSE}=2 \eps \log ( 1/\eps)/\log (1+\epsilon(1-\epsilon)p/\sigma^2),
$$ which using $k=\eps p$ as $\eps \to 0$ simplifies with negligible multiplicative error to $$\delta_{\eps, \MMSE}=2\left(k/p\right) \log (p/k)/\log (1+k/\sigma^2).$$ Note that this is the exact information-theoretic threshold for which an all-or-nothing phenomenon has been proven to hold when $\limsup_p \log k/\log p<0.5$ in \cite{reeves:2019b}.

Similarly, according to Corollary \ref{cor:AMP}, (\ref{AMPsparse1}), and \eqref{eq:snr_binary}, the limiting MSE of the AMP exhibits an all-or-nothing behavior at: $$
\delta_{\eps, \AMP}=2\left(1+ \frac{\sigma^2}{p\eps(1-\eps)} \right) \eps \log ( 1/\eps),
$$ 
which using $k=\eps p$ as $\eps \to 0$  simplifies with negligible multiplicative error to $$\delta_{\eps, \AMP}=2\left(k+\sigma^2\right)\log (p/k)/p.$$ Note that this is the exact computational threshold for a number of computationally efficient methods in the literature such as LASSO or Orthogonal Matching Pursuit (see \cite{Zadik17a,Zadik17b} for references). 
Our result suggests that the threshold corresponds to a barrier also for AMP in a strong sense.

\bibliographystyle{IEEEtran}

\bibliography{camsap19_arxiv.bbl}


\appendix
\section{Properties of RS and AMP Formulas under finite entropy}

This section describes some properties of the potential function $\cF(s)$ defined in (\ref{pote}), such as its minimum value  $\cF^*$, the upper and lower minimizers $(\overline{s}^*, \underline{s}^*)$, and the smallest stationary point $s^\AMP$, as a function of the problem parameters $(\delta, \snr, \pi)$. Throughout, we make the additional assumption that $\pi$  is a discrete distribution  with finite entropy $H = H(\pi)$. 

\begin{lemma}\label{lem:Ibounds}
The mutual information function  $I(s)$ of the single-letter channel under input distribution $\pi$ defined in (\ref{MIS}) satisfies
\begin{align}
\min\left\{ \frac{s}{2} , H \right\} -  L \le   I(s) \le  \min\left\{ \frac{s}{2} , H \right\}
\end{align}
for all $s \in [0, \infty)$, where \begin{align} \label{ell} L   := H -  I(2 H). \end{align}
\end{lemma}
\begin{proof}
Define $A(s) : =  s/2 - I(s)$. By the I-MMSE relation and the assumption of the unit variance, the derivative $A'(s) = (1 - M(s))/2$ is nonnegative, and thus $0 \le s \le 2H$ implies that  $A(0)  \le  A(s) \le A(2 H)$. Noting that $A(0) = 0$ and $A(2 H) = L$ yields $0 \le  s/2 - I(s) \le  L $ for all $s \in [0, 2H]$.  Meanwhile,  the mutual information also satisfies the upper bound $I(s) \le H$ for all $s$. Finally, because $I(s)$ is non-decreasing, $s \ge 2H$ implies that $I(s) \ge I(2H)  = H  - L$. Combining these inequalities gives the stated result. 
\end{proof}

\subsection{Minimizer of potential function}

We now consider upper and lower bounds on the minimizers of the potential function $\cF(s)$ defined in (\ref{pote}). The basic idea behind our approach is to use the following simple relations for the minimizers of $\cF(s)$: 
\begin{align}
\min_{s \in (0, t]} \cF(s) > \min_{s \in (0, \infty )} \cF(s)  \quad \implies \quad \underline{s}^* > t\\
\min_{s \in (0, \infty)} \cF(s) < \min_{s \in [t, \infty )} \cF(s)  \quad \implies \quad \overline{s}^* <  t .
\end{align}

\begin{lemma}\label{lem:sLB}
For any $t \in [2H, \infty)$, 
\begin{align}
\delta > \frac{ t + 2L}{ \log(1 + \snr)}  \quad \implies \quad \underline{s}^* >  t .
\end{align}
\end{lemma}
\begin{proof}
Observe that  $\phi(x)$ is convex and thus $\phi(y) \ge \phi(x) + (y-x) \phi'(x)$ for all $x,y \in (0,\infty)$. Noting that $\phi'(x) = 1 - 1/x$ and evaluating with $y = s/(\delta \snr)$ and $x = 1/(1 + \snr)$ leads to 
\begin{align}
\phi\left( \frac{s   }{\delta \snr }   \right) &  \ge   \log(1 + \snr)   -    \frac{ s}{ \delta}. 
\end{align}
Using this inequality to lower bound the potential function, we have
\begin{align}
 \min_{s \in [0,  t)} \cF(s)  
& \ge   \min_{s \in [0, t)}  \left\{ I(s) + \frac{ \delta}{2} \log(1 + \snr) - \frac{s}{2}   \right\}\\
& \ge  H  - L  + \frac{ \delta}{2} \log(1 + \snr) - \frac{t}{2}   >  H,
\end{align}
where the second step follows from Lemma~\ref{lem:Ibounds} and the assumption $t \ge 2H$, and the last step follows from the assumption on $\delta$. Meanwhile, using the upper bound $I(s) \le H$, we have
\begin{align}
  \min_{s \in (0,  \infty)} \cF(s)   \le  \min_{s \in (0,  \infty)} \left\{ H + \frac{\delta}{2} \phi\left( \frac{s   }{\delta \snr }   \right) \right\}  =  H.
\end{align} 
Thus,  we  can conclude that the minimum is not attained in the interval $(0, t]$. 
\end{proof}

\begin{lemma}\label{lem:sUB}
For any $t \in (0, 2H]$, 
\begin{align}
\delta <  \frac{ t -  2L}{ \log(1 + \snr)}  \quad \implies \quad \overline{s}^* < t.
\end{align}
\end{lemma}
\begin{proof}
Noting that $\phi(x) \ge 0$, we have
\begin{align*}
 \min_{s \in [t,  \infty)} \cF(s)  
& \ge   \min_{s \in [t,  \infty)}   I(s) \ge I(t) \ge \frac{t}{2} - L >   \frac{\delta}{2} \log(1 + \snr) ,
\end{align*}
where the third step follows from Lemma~\ref{lem:Ibounds} for $s=t$ and the assumption $t  \leq  2H$, and the last step follows from the assumption on $\delta$. Meanwhile, using the upper bound $I(s) \le s/2$, we have
\begin{align}
  \min_{s \in (0,  \infty)} \cF(s)   \le  \min_{s \in (0,  \infty]} \left\{ \frac{s}{2} + \frac{\delta}{2} \phi\left( \frac{s   }{\delta \snr }   \right) \right\}  =  \frac{\delta}{2} \log(1+ \snr) .
\end{align}
Thus, we can conclude that the minimum is not attained in the interval $[t, \infty)$. 
\end{proof}

\subsection{Smallest stationary point of the potential function}

\begin{lemma}\label{lem:sAMPbounds}
For any $t \in (0,\infty)$, 
\begin{align}
\delta < \sup_{s \in (0, t]}  s\left( M( s)  + \frac{1}{\snr} \right)  \quad \implies \quad  s^\AMP < t\\
\delta > \sup_{s \in (0, t]}  s\left( M( s)  + \frac{1}{\snr} \right)  \quad \implies \quad  s^\AMP > t.
\end{align}
\end{lemma}
\begin{proof}
By differentiation and the I-MMSE relation, one finds that every stationary point of $\cF(s) $ satisfies the fixed-point equation $  M(s)   -\delta / s + 1/\snr = 0$. Rearranging and solving for $\delta$, we see that $s$ is a stationary point if and only if $ \delta_\mathrm{FP}( s) = \delta$ where 
\begin{align}
\delta_\mathrm{FP}(s)  & := s\left( M( s)  + \frac{1}{\snr} \right) .
\end{align}
The function $\delta_\mathrm{FP}(s)$ is continuous with $\delta_\mathrm{FP}(0) = 0$. Therefore, if $\delta < \sup_{s \in (0, t]} \delta_\mathrm{FP}(s)$ then the equation $\delta_\mathrm{FP}(s) =\delta$ has at least one solution on $[0, t)$, and this implies that $s^\AMP < t$. Conversely, if  $\delta > \sup_{s \in (0, t]} \delta_\mathrm{FP}(s)$ then there is no solution on $[0, t)$ and this  implies that $s^\AMP >  t$. 
\end{proof}

 \section{Proofs of Main Results} 
 
 \subsection{Proof of Theorem~\ref{thm:MMSE_Thm}}

First we show that Assumption~\ref{dfn:stepass} implies that $L_\eps / H_\eps \to 0$, where we recall that $L_\eps = H_\eps - I_\eps(2 H_\eps)$. To see why, observe that
\begin{align}
L_\eps   & = \frac{1}{2} \int_0^{2H_\eps} ( 1 -  M(s))  \, \dd s  =  H_\eps  \int_0^{1} ( 1 -  M(2 H_\eps t))  \, \dd t. 
\end{align} 
For any $\eta \in (0, 1)$, the integral on the right can be upper bounded as 
\begin{align}
 \int_0^{1} ( 1 -  M(2 H_\eps t))  \, \dd t \le  \int_0^{1-\eta} ( 1 -  M(2 H_\eps t))  \, \dd t  + \eta . 
\end{align}
By Assumption~\ref{dfn:stepass}, the first term on the right-hand side converges to zero in the small-$\eps$ limit. Noting that $\eta$ can be chosen arbitrarily small establishes that $L_\eps / H_\eps \to 0$. 

We are now ready to consider the case $r \in (0, 1)$.  Fix $\eta$ such that $ r < 1- \eta < 1$ and let $t_\eps = (1  - \eta/2) 2H_\eps$. For all $\eps$ sufficiently small, we have $\eta H_\eps > 2 L_\eps $  and thus  $t_\eps  - 2 L_\eps \ge (1- \eta) 2 H_\eps >  r 2H_\eps$.  Under the assumed scaling in \eqref{eq:scaling_r}, this means that $\delta_\eps \le (t_\eps  - 2 L_\eps ) / \log(1 + \snr_\eps)$ for all $\eps$ small enough.  By Lemma~\ref{lem:sUB}, this implies that $\overline{s}^* < (1- \eta/2) 2 H_\eps$ and by Assumption~\ref{dfn:stepass}, it follows that  $M_\eps \left( \overline{s}^*_\eps \right )  \to 1$. The case $r \in (1,\infty)$ follows from a similar argument and is omitted.

\subsection{Proof of Theorem~\ref{thm:AMP_Thm}}
Consider the case $r \in (0,1)$. Fix $\eta$ such that $r < (1- \eta)^2 < 1$ and let $t_\eps = (1- \eta) 2H_\eps$. 
By \eqref{eq:Meps_cond}, for all $\eps$ sufficiently small, we have $M_\eps(s) \ge 1- \eta$ for all $s \in [0, t_\eps]$,  and thus
\begin{align}
\sup_{s \in [0, t_\eps]} s\left( M_\eps( s)  + \frac{1}{\snr_\eps} \right) & \ge \sup_{s \in [0, t_\eps]} s\left( 1  - \eta + \frac{1}{\snr_\eps} \right)\\
&  =   (1 -\eta) 2 H_\eps \left( 1  - \eta + \frac{1}{\snr_\eps} \right) \\
& >  (1 -\eta)^2 2 H_\eps \left( 1 + \frac{1}{\snr_\eps} \right).
\end{align}
The scaling \eqref{eq:scaling_rho} combined with the assumption $r < (1- \eta)^2$ means that, for all $\eps$ small enough, 
\begin{align}
\delta_\eps < \sup_{s \in [0, t_\eps]} s\left( M_\eps( s)  + \frac{1}{\snr_\eps} \right). 
\end{align}
By Lemma~\ref{lem:sAMPbounds}, this implies that $s^\AMP_\eps <  t_\eps$ and by \eqref{eq:Meps_cond}, this means that $M_\eps (s^\AMP_\eps )  \to 1$. The case $r \in (1,\infty)$ follows from a similar argument and is omitted.

\end{document}